\documentclass[12pt,oneside]{article}
\usepackage{verbatim}

\usepackage[utf8]{inputenc}
\usepackage{setspace}
\usepackage{fancyhdr}
\usepackage{tabularx}
\usepackage[letterpaper,margin=1in]{geometry}
\usepackage{mathtools} % imports amsmath
\usepackage{amsfonts}
\usepackage{amsthm}
\usepackage{MnSymbol}
\usepackage[pdftex]{graphicx}
\usepackage{diagbox}
\usepackage{xspace}
\usepackage{xcolor}
\usepackage{relsize}
\usepackage{tipa}
\usepackage{phonetic}
\usepackage{enumitem}

\newtheorem{theorem}{Theorem}
\newtheorem{proposition}{Proposition}
\newtheorem{corollary}{Corollary}
\newtheorem{lemma}{Lemma}

\theoremstyle{definition}
\newtheorem{example}{Example}
\newtheorem{definition}{Definition}
\newtheorem{remark}{Remark}

\DeclareMathOperator{\divi}{|}

\newcommand{\NatZero}{\mathbb{N}_{0}} 
\newcommand{\N}{\mathbb{N}} 

\newcommand{\Primes}{\mathbb{P}}

\newcommand{\kary}{$k$-ary\xspace}

\renewcommand{\leq}{\leqslant}

\newcommand{\1}{\{1\}}

\newcommand{\A}{\mathbb{A}}
\newcommand{\T}{\Theta}
\newcommand{\Pow}{\mathcal{P}}
\newcommand{\Ang}{\mathcal{A}}

\newcommand{\makePersonName}[2]{\newcommand{#1}{#2\xspace}}

\makePersonName{\Mobius}{M\"{o}bius}
\makePersonName{\Toth}{T\'{o}th}

\newcommand{\perf}{P}
\newcommand{\prim}{P}

\newcommand{\Rat}{\rm{Rat}}
\newcommand{\Aname}{\text{paradigmatic }}
\renewcommand{\S}{\Sigma}
\newcommand{\makenest}[3]{\newcommand{#1}[1]{\ensuremath{\left#2##1\right#3}}}

\makenest{\braced}\{\}
\makenest{\set}\{\}
\makenest{\parens}()
\makenest{\chevroned}<>
\makenest{\bracketed}\[\]
\makenest{\norm}\|\|
\makenest{\abs}||

\usepackage{authblk}

\title{On a General Class of $A$-functions}

\author[1]{Joseph Vade Burnett}
\author[2]{Alex Taylor}
\affil[1]{The University of Texas at Dallas, Joseph.Burnett@utdallas.edu}
\affil[2]{The University of Illinois at Urbana-Champaign, alext3@illinois.edu}

\begin{document}

\maketitle

\begin{abstract}
    Expanding upon recent work, a new class of $A$-functions is introduced that can be viewed as an appropriate generalization of the class of regular $A$-functions, the class of structured $A$-functions, and the class of perfect $A$-functions. Algebraic results are discussed and analysis is performed to deduce new results concerning rational arithmetical functions.
\end{abstract}

\section{Introduction}
%Update 11/27/2020: Will need to work on this.

\subsection{$A$-functions}
Denote by $D(n)$ the set of positive divisors of $n$. Let $A: \N \rightarrow \Pow(\N)$ be any function such that $1 \in A(n) \subseteq D(n)$. Then $A$ is called an \textbf{$A$-function}.

Examples of $A$-functions have appeared in the literature for many years, but usually in different forms. To each $A$-function $A$ a relation $\divi_A$ may be associated, where $d \divi_A n$ if and only if $d \in A(n)$. In this context we may call such $d$ ``$A$-divisors of $n$'' and discuss the ``$A$-divisibility relation''.

Denote by $\Ang$ the set of complex-valued arithmetical functions. If $A$ is an $A$-function, an arithmetical convolution-type operation $*_A$ may be associated to it, where for $f,g \in \Ang$, 

\begin{equation}
    (f*_Ag)(n) := \sum_{d \in A(n)} f(d) g\parens{\frac{n}{d}}.    
\end{equation}

In this case $*_A$ is called the ``$A$-convolution'', and we may discuss algebraic properties of this operator. Since the term ``$A$-function'' was only recently introduced, yet satisfies a need to be able to refer to divisibility relations and arithmetical convolutions in the abstract, we shall refer to such divisibility relations and arithmetical convolutions appearing in the literature as $A$-functions, and speak of them in terms of $A$-functions in order to have a uniform discussion.

\subsection{Previous work}

One of the first $A$-functions to be introduced was done so via the unitary divisibility relation of Vaidyanathaswamy \cite{vaid}. Subsequently, various works were produced which studied and generalized the unitary divisibility relation and associated convolution (see \cite{Cohen1,hauk,Ken,tothB}). In several cases, entire classes of $A$-functions were introduced, viewing the usual divisibility relation (which corresponds to the $A$-function $D$) and the unitary divisibility relation as the first two elements of a sequence of $A$-functions that preserved some property or theme (see \cite{alla,us1,Cohen2,Cohen3,hauk1,sury}).

Certain classes of $A$-functions have been the subject of many manuscripts devoted to exploring their properties. The regular $A$-functions of Narkiewicz \cite{nark} in particular have been heavily investigated and generalized (see \cite{us3,hanu,hauk2,Mc,rama,tothP,toth,tothS,yocom}).

Other generalizations include the generalizations of the infinitary divisibility relation of Cohen (see \cite{us2,us4,Cohen2,Cohen3,LS}), which will be actively discussed in this work, and the generalization of convolutions to convolutions involving weight functions (see \cite{dave,foti}).

\subsection{Purpose}

In this work, we describe the paradigm through which we view many of the $A$-functions in the literature. In particular, the class of regular $A$-functions, the class of structured $A$-functions, and the class of perfect $A$-functions are portrayed as being the three sides of a diagrammatic triangle (see figure below) whose vertices are the $A$-functions $D$, $D^1$, and $\T$, corresponding to the Dirichlet, Unitary, and Infinitary relations/convolutions, respectively. 

\begin{center}\includegraphics[scale=.2]{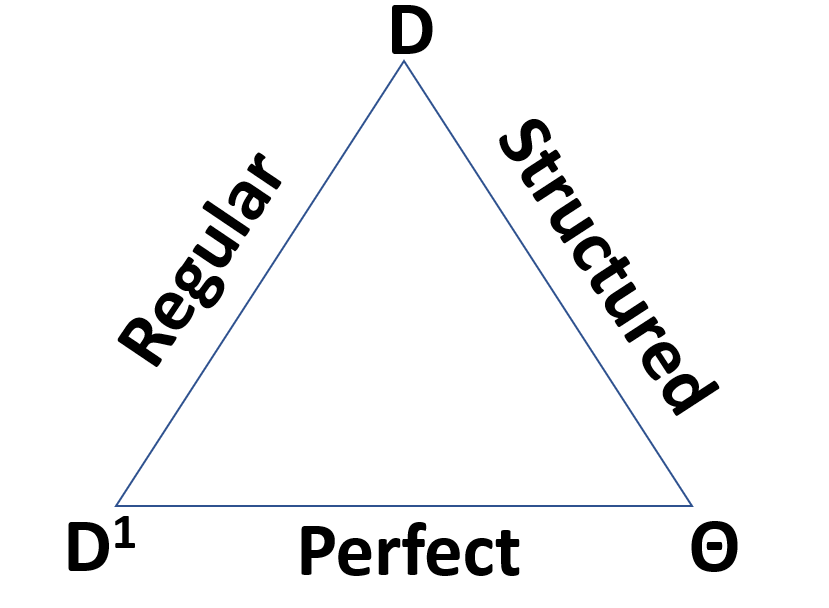}\end{center}

One may then view the regular, structured, and perfect $A$-functions as being those $A$-functions that preserve properties common to $D$ and $D^1$, to $\T$ and $D$, and to $D^1$ and $\T$, respectively.

We introduce a new broad class of $A$-functions that contains as special subclasses the regular, structured, and perfect $A$-functions. From this broad class we extract a subclass, which we call the \textbf{\Aname} $A$-functions, whose members are $A$-functions possessing many properties common to the regular, structured, and perfect $A$-functions. Algebraic properties pertaining to the corresponding arithmetical convolution, numerical factorization properties tied to the concept of multiplicativity, and analytic properties resulting from a natural metric on $\Ang$ are explored.

\subsection{Outline}

In Section \ref{Prelims} we outline some preliminary definitions and results concerning both $A$-functions in the general and in specific instances when $A$ is regular or structured or perfect. The notion of the \textit{iterate} is recalled for later use, and we discuss various arithmetical functions which will appear later, including rational arithmetical functions and a natural metric on $\Ang$. 

Special attention is given to the $A$-function $\T$, corresponding to the infinitary divisibility relation/convolution, and in particular we introduce a new equivalent definition of $\T$. Finally, we state the properties common to the regular, structured, and perfect $A$-functions in order to guide the next section.

In Section \ref{complete}, we introduce our broad class of $A$-functions, termed \textbf{complete $A$-functions}, in terms of certain generalized multiplicativity conditions. A unique decomposition of each $n \in \N$ based off of $A$ follows, enabling us to define the \textbf{paradigmatic} $A$-functions in a natural manner. It is shown that \Aname $A$-functions correspond to commutative, associative $A$ convolutions and that they may be used to characterize all complete $A$-functions.

We introduce the concept of the ``maximal class'' of an arithmetical function $f$ with $f(1) = 1$, which is related to a concept of Ryden \cite{ryden}. We prove Ryden's theorem using our new notation and results, and deduce that all \Aname $A$-functions may be realized as the maximal class of some arithmetical function. We further show how every complete $A$-function may be associated to a perfect $A$-function. Examples abound throughout this section, including a special focus on the $A$-function $G$ from \cite{us3}.

Finally, in Section \ref{Results} we bring together all of our tools in order to state a result with an analytic flavor which is a generalization of the following statement: Every multiplicative arithmetical function may be ``approximated'' using rational arithmetical functions. Here, ``approximation'' occurs with respect to the metric of Schwab and Silberg \cite{schw}. %Several examples are given and a new $A$-function with potential for generalization is introduced.

\section{Preliminary Definitions and Results}\label{Prelims}

\subsection{Properties $A$-functions can possess}

Let $A \in \A$. $A$ is called \textbf{simple} if $1 \in A(n)$ for all $n$ and \textbf{reflexive} if $n \in A(n)$ for all $n$. $A$ is called \textbf{symmetric} if $d \in A(n) \Rightarrow \frac{n}{d} \in A(n)$. $A$ is \textbf{transitive} if $A(d) \subseteq A(n)$ whenever $d \in A(n)$, which is to say that $\divi_A$ is transitive as a relation. If $A(n) = \{1,n\}$ ($n \neq 1$) then $n$ is called \textbf{$A$-primitive}; 1 is not considered $A$-primitive for any $A$. We denote the set of $A$-primitive elements by $\Primes_A$, with $\Primes := \Primes_D$ the set of prime numbers. Note that $\Primes \subseteq \Primes_A$ as long as $A$ is reflexive.

$A$ is said to be \textbf{multiplicative} if $A(mn) = A(m) \cdot (n) := \{ab \divi a \in A(m), b \in A(n)\}$ whenever $(m,n) = 1$. Since $(m,n) = 1 \Leftrightarrow m \in D^1(mn)$, it is easy to see that $A$ is multiplicative if and only if $m \in D^1(mn) \Rightarrow A(mn) = A(m) \cdot A(n)$. This has been generalized by \cite{us4,Mc,yocom}. In (\cite{us4}), the authors defined an $A$-function $A$ to be \textbf{class-$B$}, for $B \in \A$, if $m \in B(mn) \Rightarrow A(mn) = A(m) \cdot A(n)$. We denote the set of all class-$B$ $A$-functions by $C(B)$. 

$A$ is called \textbf{homogeneous} if $A$ is multiplicative and for every $p_1,p_2 \in \Primes$ and every $0 \leq b \leq a$, $p_1^b \in A(p_1^a) \Leftrightarrow p_2^b \in A(p_2^a)$. Examples of homogeneous $A$-functions include the \kary $A$-functions of Cohen \cite{Cohen2}, the $k$-factorization $A$-functions of Litsyn and Shevelev \cite{LS}, and the higher-order $A$-functions of Alladi \cite{alla}.

On the other hand, we call $A$ a \textbf{cross-convolution} if it is multiplicative but not homogeneous. The concept of cross-convolution in the context of $A$-functions is due to \Toth \cite{toth}, and was formalized and generalized in \cite{us2}. If we let $\{A_{\lambda}\}_{\lambda \in \Lambda}$ be an indexed set of homogeneous $A$-functions and let $f:\Primes \rightarrow \Lambda$ be a function, then we may define $A_f$ to be the multiplicative $A$-function such that $A_f(p^a) = A_{f(p)}(p^a)$ for all $p \in \Primes$ and all $a \in \N$, with $A_f(1) = \1$ by convention. We then call $A_f$ the \textbf{$f$-cross convolution $A$-function with respect to $\{A_{\lambda}\}_{\lambda \in \Lambda}$}. If $f$ is a constant function, then $A_f$ is homogeneous.

The original cross-convolution $A$-functions of \Toth \cite{toth}, the structured $A$-functions of \cite{us2}, and the mixed-multiplicative $A$-functions of Litsyn and Shevelev \cite{LS} are examples of families of cross-convolution $A$-functions appearing in the literature.

\subsection{The iterate}

The divisibility relation corresponding to the unitary $A$-function $D^1$ was first introduced by Vaidyanathaswamy \cite{vaid}, who noted properties that it possessed which were similar to the usual divisibility relation, and who essentially initiated a search for other such divisibility relations. In particular, this led to the introduction of entire families of $A$-functions having $D$ and $D^1$ as elements:%\footnote{Because the concept of an $A$-function is old, and because each $A$-function may be associated to a relation on $\N$ and an arithmetical convolution on the set of arithmetical functions, we will frequently refer to the appropriate relations or convolutions appearing in the literature in terms of the $A$-functions that may be associated to them, as this is our preferred paradigm for viewing these objects.}:

\begin{itemize}
    \item The Suryanarayana \kary $A$-functions $S_k$ of \cite{sury}, where $d \in S_k(n)$ if and only if $d \divi n$ and the largest $k^{th}$-power divisor of $d$ and $\frac{n}{d}$ is 1.
    \item The regular $A$-functions of Narkiewicz, discussed in Section \ref{Regular}.
    \item The higher-order $A$-functions $A_k$ of Alladi \cite{alla}, which are multiplicative and are recursively defined such that $p^b \in A_k(p^a) \Leftrightarrow A_{k-1}(p^b) \cap D(p^{a-b}) = \1$, where $A_0$ is $D$ and $A_1$ becomes $D^1$.
    \item The \kary $A$-functions of Cohen \cite{Cohen2} (which were inspired by Suryanarayana's introduction of the bi-unitary $A$-function $D^2$ and the recursive approach taken by Alladi), which we denote by $D^k$, where $p^b \in D^k(p^a)$ if and only if $D^{k-1}(p^b) \cap D^{k-1}(p^{a-b}) = \1$, where $D^0 := D$ and $D^1$ is the unitary divisibility relation.
\end{itemize}

Taking inspiration from the final family of $A$-functions, the \kary $A$-functions of Cohen, the authors of \cite{us2} introduced the following definition for what they called the \textbf{iterate} of an $A$-function:

\begin{definition}
    Let $A \in \A$ be simple. Then the \textbf{iterate} of $A$, $A^1$, is the $A$-function such that $d \in A^1(n) \Leftrightarrow A(d) \cap A\parens{\frac{n}{d}} = \1$. For $k > 1$, the $k$-fold iterate of $A$, $A^k$, is the $A$-function such that $A^k = (A^{k-1})^1$.
\end{definition}

The \kary $A$-functions $D^k$ of Cohen are thus defined in terms of the $k$-fold iterate. By applying the iterate recursively to an $A$-function $A$, we may obtain a family of $A$-functions $\{A^k\}$ analogous to the \kary $A$-functions of Cohen. General results for arbitrary $A$ have not appeared in the literature.

\begin{remark}
    In using the notation $A(d) \cap A\parens{\frac{n}{d}}$, the authors of \cite{us2} have circumvented an awkward process present in much of the literature, whereby in order to define a new family of $A$-functions it was necessary to first define an analogue of the GCD-function and express which divisors $d$ of $n$ were to be included in the set $A(n)$ on the basis of some condition on this analogue of the GCD-function. (cf. \cite{us1,Cohen2,hauk})
    
    Rather than define entire families of auxiliary GCD-functions whose sole purpose is to serve as indicator functions for membership in $A(n)$ for each $n$, we shall continue in the convention of \cite{us2} and use conditions on specific values of $A$-functions themselves to define new $A$-functions, where needed, as is done above.
\end{remark}

\subsection{Arithmetical Functions}

Denote by $\Ang$ the set of complex-valued arithmetical functions, and if $X$ is some set by $\Ang_X$ the set of $X$-valued arithmetical functions. Typically, $X$ is taken to be a set with algebraic structure allowing for the addition and multiplication of its elements; in our case, we will be utilizing $X = \mathbb{R}$, the set of real numbers in Section \ref{complete}. One may think of $\Ang$ as either a set of arithmetical functions or as the space $\mathbb{C}^{\infty}$ of sequences of complex numbers, the latter being the natural context in which the Dirichlet convolution can be realized.

Throughout this paper we will be employing several arithmetical functions commonly appearing in the literature, including the Euler totient function $\varphi$, the divisor function $\tau$, and the \Mobius function $\mu$. Let $\gamma(n)$ be the \textbf{multiplicative core} of $n$, i.e., \[\gamma(n) = \prod_{p \divi n}p.\] We shall also make use of the function $u$, where $u(n) = 1$ for all $n$, the function $I$, where $I(n) = n$ for all $n$, and the function $\iota$, the identity under the Dirichlet convolution, where

\[\iota(n) = \begin{cases}1&\text{ if }n=1 \\ 0&\text{otherwise.} \end{cases}\] 

If $A$ is simple and reflexive, a result of Narkiewicz states that $\iota$ is the unique identity element with respect to the $A$-convolution $*_A$ in the sense that for all $f \in \Ang$,

\begin{equation}
    \iota * f = f*\iota = f.   
\end{equation}

Narkiewicz \cite{nark} also remarked that for all simple and reflexive $A$-functions, the corresponding $A$-convolution is such that every arithmetical function $f$ with $f(1) \neq 0$ possesses a unique left inverse and a unique right inverse with respect to the $A$-convolution $*_A$; that is, for $f \in \Ang$ there is a unique $f^r$ and $f^l$ in $\Ang$ such that

\begin{equation}
    \begin{aligned}
        f *_A f^r = \iota \\ f^l *_A f = \iota.
    \end{aligned}
\end{equation}

In the case where $A$ is symmetric, $*_A$ is commutative and hence the unique left and right inverses of a function coincide.

In keeping with the notation of \cite{us3}, for a simple and reflexive $A$-function $A$, we define $\mu_A$ to be the left inverse of $u$ under the $A$-convolution $*_A$. Note that this generalizes the definition of the usual \Mobius function to all applicable $A$. In this paper, $\mu_A$ will only appear in the context of symmetric $A$-functions, so the specification of $\mu_A$ as the right inverse is redundant, as in this case there will be only one \Mobius function corresponding to $A$.

We let $\tau_A$ denote the $A$-convolution of $u$ with itself. Combinatorically speaking, $\tau_A$ counts the number of ``$A$-divisors'' of $n$, which is just the size of the set $A(n)$.

Just as $A$-functions can belong to a given class, generalizing the concept of multiplicativity of $A$-functions, the conept of multiplicativity of arithmetical functions has been generalized on several occasions. Recall that an arithmetical function $f$ is \textbf{multiplicative} if $f$ is not identically zero and $f(mn) = f(m)f(n)$ whenever $(m,n) = 1$, which is to say whenever $m \in D^1(mn)$. Yocom \cite{yocom} called an arithmetical function $A$-multiplicative, for a regular $A$-function $A$, if $f$ is not identically zero and $m \in A(mn) \Rightarrow f(mn) = f(m)f(n)$. Burnett, Osterman, and Lewandowski \cite{us4} defined an arithmetical function $f$ to be \textbf{class-$A$} for an arbitrary $A$-function $A$ if $f$ is not identically zero and $m \in A(mn) \Rightarrow f(mn) = f(m)f(n)$. We then refer to $A$ as an \textbf{arithmetical class} of $f$. As this is a more generalized version of the concept of multiplicativity, we shall adopt the notation of \cite{us4} for our present work. Thus, we denote by $c(A)$ the set of class-$A$ arithmetical functions. Then $c(D^1)$ is the set of all multiplicative functions and $c(D)$ is the set of all completely multiplicative functions.

\subsection{A metric on $\Ang$}

We reformulate a definition of Schwab and Silberg \cite{schw} for our purposes:

\begin{definition}
    Let $d:\Ang \times \Ang \rightarrow \mathbb{R}$, \begin{equation}\label{eq:metric}
        d(f,g) = 
        \begin{cases}
            0 & \text{ if }f = g \\
            \frac{1}{\min (k \in \N \divi f(k) \neq g(k))} & \text{ if }f \neq g
        \end{cases}
    \end{equation}
\end{definition}

It was shown that \eqref{eq:metric} defines a metric on $\Ang$, and it is not difficult to show that this metric is complete. In Section \ref{Results} we will perform analysis on subsets of $\Ang$ with respect to this metric.

\subsection{Rafts and $A$-rafts}

Suppose $f \in \Ang$ is the Dirichlet convolution of some number of completely multiplicative arithmetical functions and some number of inverses (with respect to the Dirichlet convolution); that is, \begin{equation}\label{eq:raft}
    f = f_1 * \cdots * f_r * (f_{r+1})^{-1} * \cdots * (f_{r+s})^{-1},
\end{equation}

where $(f_i)^{-1}$ is the inverse of $f_i$ under the Dirichlet convolution. Functions of this form have traditionally been called \textbf{rational arithmetical functions}, and were introduced by Vaidyanathaswamy \cite{vaid}. Because of the unwieldiness of this expression, and since we intend to express generalizations thereof, we shall abbreviate a \textbf{r}ational \textbf{a}rithmetical \textbf{f}unc\textbf{t}ion to be, simply, a \textbf{raft}. 

Suppose instead of being completely multiplicative (i.e., class-$D$), we restrict the $f_i$ in equation \eqref{eq:raft} to be class-$A$ for some simple, symmetric, and reflexive $A \in \A$, and we similarly replace the usual Dirichlet convolution with the $A$-convolution:

\begin{equation}\label{eq:araft}
    f = f_1 *_A \cdots *_A f_r *_A (f_{r+1})^{-1} *_A \cdots *_A (f_{r+s})^{-1},
\end{equation}

where $(f_i)^{-1}$ is the inverse of $f_i$ under the $A$-convolution. Then we refer to $f$ as an \textbf{$A$-raft}. Typically, if $s = 0$ a raft (resp. $A$-raft) is called \textbf{$r$-ic} \cite{carroll:1975:raft} (resp. \textbf{$A$-$r$-ic}).

$A$-rafts have been studied in the literature before in the context of $A$ being a regular $A$-function; see \cite{hauk2} for a derivation of some relevant analogues of rafts in the context of $A$-rafts. For the purpose of this paper, $A$ will always be chosen in such a way that $*_A$ is associative, justifying the lack of parentheses in \eqref{eq:araft}. 

Corresponding to each $A$-function we may define $\Rat(A)$ to be the set of all $A$-rafts. We shall show in Section \ref{Results} that under certain conditions, $\Rat(A) = c(B)$ for an appropriately chosen $B$.

For further properties and identities involving rafts, see \cite{lao}.

\subsection{The infinitary $A$-function $\T$}

The infinitary $A$-function of Cohen, which we shall denote by $\T$ (keeping with the notation of \cite{us2,us4}), is of central importance to this work and is more difficult to define than $D$ and $D^1$, the other two $A$-functions that we will commonly reference.

There are many equivalent ways of defining $\T$ -- see \cite{us2,us4,Cohen2,Cohen3,LS} for a variety of ways $\T$ can be defined. We will discuss here a new equivalent way to define $\T$.

\begin{proposition}\label{infinitary}
    Suppose $S \subseteq \N$ is a subset of $\N$ such that every integer $n \in \N$ may be represented uniquely as a product of all the elements of some finite subset $S_n \subseteq S$. Then if $S' := \{p^{2^x} \divi p \in \Primes, x \in \NatZero\}$, then $S = S'$.
\end{proposition}

\begin{proof}
    Clearly $1 \notin S$ and $\Primes \subseteq S$, as every prime number must be represented as itself. Furthermore, $p^2 \in S$ for all primes $p$, since there is no non-singleton subset of $\N$, the product of whose elements is the representation of $p^2$ as $p \cdot p$. $p^3$ is represented uniquely as $p \cdot p^2$, and $p^4$ must represent itself by the same reasoning as for $p^2$. Continuing in this manner, we can see that $S' \subseteq S$.
    
    To see that $S \subseteq S'$, note that every integer $n \in \N$ may be represented uniquely as a product of elements of $S'$ in the following manner:
    
    \begin{enumerate}
        \item First, decompose $n$ according to its prime decomposition: \[n = p_1^{a_1} \cdots p_l^{a_l}\]
        \item Then, write each $a_i$ in terms of its binary representation as powers of 2: \[a_i = \sum_{j = 0}^{\infty} b_{i,j}2^j,\] where $b_{i,j} \in \{0,1\}$.
        \item Then 
        \begin{equation}
            n = p_1^{b_{1,0}2^0 + b_{1,1}2^1 + \dots} \cdots p_k^{b_{l,0}2^0 + b_{l,1}2^1 + \dots} = q_1 \cdots q_r,
        \end{equation}
        where each $q_k$ is one of the $p_i^{b_{i,j}2^j}$ for which $b_{i,j} = 1$.
    \end{enumerate}
    
    Hence, $S = S'$ and we are done.
\end{proof}

The set $S$ may then be identified with the $I$-components of Cohen \cite{Cohen2,Cohen3} and hence with the $\T$-primitive elements $\Primes_{\T}$ as mentioned in \cite{us2}. We may use $S$ to define $\T$ as follows:

\begin{definition}
    For all $n \in \N$, we say that $d \in \T(n)$ if $d$ may be represented as a product of some subset of the elements in $S_n$.
\end{definition}

Observe that if $d \in \T(n)$ then $\T(d) \subseteq \T(n)$ as $S_d \subseteq S_n$. Since $S_{\frac{n}{d}} = S \setminus S_d$, it follows that $d \in \T(n) \Leftrightarrow \T(d) \cap \T\parens{\frac{n}{d}} = \1$. This tells us that the iterate of $\T$ is $\T$, a fact observed in \cite{us2}.

\subsection{The regular $A$-functions}\label{Regular}

Narkiewicz \cite{nark} formally introduced the concept of $A$-functions for the purpose of analyzing a specific class of $A$-functions called \textbf{regular} $A$-functions.

Narkiewicz defined an $A$-function to be regular if the following conditions hold:

\begin{enumerate}
    \item The ring $(\Ang,+,*_A)$ is associative, commutative, and possesses a unit element.
    \item The conovlution $*_A$ preserves multiplicativity of arithmetical functions.
    \item For all prime powers $p^a$, $\mu_A(p^a) \in \{0,-1\}$.
\end{enumerate}

Regular $A$-functions defined in this manner are difficult to work with. Fortunately, Narkiewicz and others \cite{us3,nark,toth} have developed more workable definitions, which we state as follows:

\begin{proposition}\label{reg}
    The $A$-function $A$ is regular if and only if
    \begin{enumerate}
        \item $A$ is multiplicative.
        \item For every $p^a$ there is a unique integer $r_A(p^a)$ such that $A(p^a) = \{1,p^{r_A(p^a)},p^{2 \cdot r_A(p^a)}, \dots, p^a\}$, and if $p^b \in A(p^a)$ then $r_A(p^b) = r_A(p^a)$.
    \end{enumerate}
\end{proposition}

For any reflexive $A$-function $A$ and any prime power $p^a$, we may define $r_A(p^a)$ to be the smallest element of $A(p^a)$ larger than 1. This extends the original definition of $r_A$ from \cite{nark} to all $A$-functions relevant to this paper.

Examples of regular $A$-functions include $D$ and $D^1$, whose properties are well-known, the cross-convolutions of \Toth \cite{toth}, and the so-called ternary $A$-function (not to be confused with the tri-unitary $A$-function, $D^3$) of \cite{gaw}, defined to be the regular $A$-function such that 

\begin{equation}\label{rapa}
    r_A(p^a) = 
    \begin{cases}
        1 & \text{ if } a \text{ is odd or a multiple of } 4 \\
        2 & \text{ if } a \text{ is twice an odd number}.
    \end{cases}
\end{equation}

Notably, $\T$ is not regular. This fact can be seen any number of ways, the simplest of which is by observing that $\mu_{\T}(p^3) = 1$, which disagrees with the condition that $\mu_A(p^a) \in \{0,-1\}$.

The regular $A$-functions possess special combinatorial properties with respect to the corresponding $A$-convolution. Their main property is best exhibited in the following result, which can be seen, for example, as a consequence of Theorem 3 of \cite{us3}:

\begin{theorem}\label{count}
    Let $\varphi_A := \mu_A *_A I$. Then $\varphi_A(n)$ counts the number of positive integers $k \leq n$ such that $D(k) \cap A(n) = \1$.
\end{theorem}

The regular $A$-functions are not the only $A$-functions with this property. In \cite{us3}, the authors introduced a more general class of $A$-function which they called \textbf{semi-regular}. A semi-regular $A$-function satisfies the following conditions:

\begin{enumerate}
    \item $A$ is simple, reflexive, and multiplicative.
    \item $A(p^a)$ contains exactly one $A$-primitive element for all prime powers $p^a$.
    \item $p^{r_A(p^a)}$ is $A$-primitive for all prime powers $p^a$.
\end{enumerate}

Theorem 5 of \cite{us3} guarantees that the semi-regular $A$-functions are the only multiplicative $A$-functions for which Theorem \ref{count} holds. Although our present work does not deal with semi-regular $A$-functions in particular, we will point to one example in particular of a semi-regular $A$-function with interesting properties in Section \ref{complete}.

We introduce the following new characterization of regular $A$-functions:

\begin{theorem}
    For $A \in \A$, $A$ is regular if and only if $A$ is semi-regular and $*_A$ is associative and commutative.
\end{theorem}

To show this, we borrow a very useful lemma from Narkiewicz:

\begin{lemma}[Associativity condition]\label{narkasso}
    Suppose $A \in \A$. Then $*_A$ is associative if and only if the following two conditions are equivalent:
    \begin{enumerate}
        \item $d \in A(m)$ and $m \in A(n)$
        \item $d \in A(n)$ and $\frac{m}{d} \in A\parens{\frac{n}{d}}$.
    \end{enumerate}
\end{lemma}

Note that $*_A$ is commutative if and only if $A$ is symmetric (cf. \cite{nark}, 2.ii). We now prove the theorem.

\begin{proof}
    The forward implication is immediate and follows from the definition of regularity in Narkiewicz \cite{nark} and Proposition 1 of \cite{us3}. To see the reverse implication, observe that if we can show condition 2 of Proposition \ref{reg} then our result follows.
    
    To this end, let $p$ be an arbitrary prime. For $a = 1$ and $a = 2$, there is nothing to show. Assume that we have demonstrated condition 2 of Proposition \ref{reg} for all $a < a_0$. Then for $a_0$, let $r_0 := r_A(p^{a_0})$. If $r_0 = a_0$ we are done. Otherwise, $r_0 < a_0$ and so $p^{a_0-r_0} \in A(p^{a_0})$ is a non-trivial statement. Observe that $A$ is transitive (since associativity of $*_A$ implies transitivity), so $A(p^{a_0-r_0}) \subseteq A(p^{a_0})$. Hence, $r_A(p^{a_0-r_0}) = r_0$, and so $\{1,p^{r_0},p^{2r_0},\cdots,p^{a_0}\} \subseteq A(p^{a_0})$. 
    
    To see that no other elements are present in $A(p^{a_0})$, note that any $p^b \in A(p^{a_0})$ such that $b$ is not a multiple of $r_0$ will carry with it (by transitivity) a different $A$-primitive element $p^{r_1} \in A(p^{a_0})$, contradicting semi-regularity. Thus, by induction, we are done.
\end{proof}

Thus, we may think of regular $A$-functions as those multiplicative $A$-functions for which Theorem \ref{count} holds true, whose $A$-convolutions are associative and commutative.

\subsection{The structured $A$-functions}

The structured $A$-functions were first introduced by Burnett and Osterman \cite{us2} as a generalization of the infinitary $A$-function. Originally, two distinct but equivalent constructions were provided. %We will introduce here a third equivalent construction that will help in visualizing the structured $A$-functions.

First, we must recall several definitions from \cite{us2}, foremost of which is the concept of the Cohen Triangle of $A$ at $p$.

For a multiplicative $A$-function $A$ and a prime number $p$, the Cohen Triangle of $A$ at $p$ is an infinite lower triangular binary array $C_{A,p} = \{c_{ab}\}$, indexed from zero, such that \[c_{ab} = 
\begin{cases}
    1 & p^b \in A(p^a) \\ 0 & \text{ otherwise.}
\end{cases} \]

In general, if $A$ is homogeneous then $C_{A,p_1} = C_{A,p_2}$ for all primes $p_1,p_2 \in \Primes$ (providing an alternative definition for homogeneity), so we refer to \textit{the} Cohen Triangle of $A$ in this case. Just as a cross convolution may be viewed as being formed out of homogeneous $A$-functions, we may view the Cohen Triangle of $A$ for various homogeneous $A$ to be the ``building blocks'' of the Cohen Triangles of cross-convolutions.

The concept of the Cohen Triangle of $A$ was first introduced by Cohen \cite{Cohen2}, who displayed progressively larger upper left submatrices that provided a visualization of the manner in which the \kary divisibility relations ``tended towards'' the infinitary divisibility relation. Any finite lower triangular matrix $\gamma$ consisting entirely of 1's will be called a \textbf{structure}.

\begin{definition}
    Let $A \in \A$ be homogeneous. We say that $A$ is (homogeneously) \textbf{$N$-finitely structured} for $N \in \N$ if there exist structures $\gamma_1,\gamma_2,\dots,\gamma_N$ such that \[C_{A} =  C_{D} \otimes \gamma_N \otimes \cdots \otimes \gamma_2 \otimes \gamma_1,\] where $\otimes$ denotes the Kronecker product of matrices. We allow $N$ to be 0, making $D$ a structured $A$-function.
    
    If, on the other hand, there exists an infinite sequence of structures $\gamma_1,\gamma_2,\cdots$ and an infinite sequence of $A$-functions $A_1,A_2,\dots$ such that
    \[\begin{aligned}
        C_{A} &= C_{A_1} \otimes \gamma_1 \\
        C_{A_1} &= C_{A_2} \otimes \gamma_2 \\
        \vdots
    \end{aligned} \] then we call $A$ (homogeneously) \textbf{infinitely structured}. If $A$ is any cross-convolution of a set of homogeneously finite or infinite structures, we simply call $A$ \textbf{structured}.
\end{definition}

Examples of structured $A$-functions include $D$, $\T$, and the $k$-factorization $A$-functions of Litsyn and Shevelev $\T_k$. In general, every structured $A$-function may be viewed as a cross-convolution of homogeneous structured $A$-functions, each of which is uniquely determined by a so-called \textbf{structure sequence} $\{s_1,s_2,\cdots\}$, which may be finite or infinite, where $s_k$ is the size of the structure $\gamma_k$.

In the remarks following Definition 2.3 of \cite{us2}, the concept of the \textbf{character} of an $B$-primitive element $q$ was introduced for structured $A$-functions $B$. We extend this to general $A$-functions:

\begin{definition}
    If $A \in \A$ is reflexive with $\Primes_A$ the set of $A$-primitive elements, then for each $q \in \Primes_A$ let $\chi_A(q)$ denote the least integer $a$ larger than 1 such that $q^a \in \Primes_A$, or $\infty$ if no such power exists.
\end{definition}

In \cite{us2}, Burnett and Osterman demonstrated

\begin{theorem}[Proposition 2.5 of \cite{us2}]
    If $B$ is a structured $A$-function and $n \in \N$, then there exists a unique decomposition of $n$ into the $B$-primitive elements, \begin{equation}
        n = q_1^{c_1}\cdots q_l^{c_l},
    \end{equation}
    such that $c_i < \chi_A(q_i)$ for all $i$.
\end{theorem}

The following algebraic properties of structured $A$-functions were also demonstrated (cf., \cite{us2}):

\begin{theorem}
    Suppose $B$ is a structured $A$-function. Then 
    \begin{enumerate}
        \item (Theorem 4.3, \cite{us2}) both $(c(B^1),*_B)$ and $(c(B^1),*_{B^1})$ form Abelian groups.
        \item (Corollary 4.6, \cite{us2}) $*_B$ preserves $c(B)$ if and only if $B = \T$
    \end{enumerate}
\end{theorem}

In Section \ref{Results} we will prove an analogue of Theorem 4.3 of \cite{us2} for a much more general class of $A$-functions.

\subsection{The perfect $A$-functions}

In Theorem 10 of \cite{Cohen3}, Cohen observed that the $A$-function $\T$ has the special property of the infintary convolution $*_{\T}$ preserving $c(\T)$ -- that is, it is an instance of an $A$-function $A$ preserving its own arithmetical class. Corollary 4.6 of \cite{us2} makes it clear that this property is not shared among any other structured $A$-function. Yet, it appears to hold for $D^1$ as well, since the unitary convolution preserves multiplicativity of arithmetical functions (i.e., $*_{D^1}$ preserves $c(D^1)$). Furthermore, it can be shown that $(c(\T),*_{\T},\cdot)$ and $(c(D^1),*_{D^1},\cdot)$ are both rings (here $\cdot$ is usual multiplication).

In Theorem 3 of \cite{us4}, the authors completely characterized all $A$-functions with this property: $A$ is such that $(c(A),*_A,\cdot)$ forms a ring if and only if $A$ is \textbf{perfect}, which is to say that $A \in C(A)$ and $A \subseteq A^1$.

It is furthermore demonstrated that all the rings created in this way are isomorphic to $(\Ang,+,\cdot)$.

Examples of perfect $A$-functions include $D^1$ and $\T$. Corollary 6 of \cite{us4} demonstrated that if $B$ is structured, then $B^1$ is perfect. Interestingly, $D^3$ is perfect, and it is currently unknown if $D^{2k-1}$ is perfect for all $k \in \N$. For an example of a perfect $A$-function which is not multiplicative, Example 7 of \cite{us4} gives us the $A$-function $Z$, where $Z(n) = \{1,n\}$ for all $n$. Note that $c(Z)$ consists of all arithmetical functions $f$ for which $f(1) = 1$.

In \cite{us4}, the authors deduced further results for general $A$-functions. Call a symmetric $A \in \A$ \textbf{split} if $m \in A(mn)$, $a \in A(m)$, and $b \in A(n)$ means $a \in A(ab)$. Then

\begin{theorem}[Lemma 5 of \cite{us4}]
    If $A \in C(A)$ and $A$ is split, then $*_A$ is associative.
\end{theorem}

\subsection{Connections between regular, structured, and perfect $A$-functions}

First, we shall demonstrate a result that reveals to us a paradigm by which we may connect regular, structured, and perfect $A$-functions.

\begin{theorem}
    The only regular and structured $A$-function is $D$. The only structured and perfect $A$-function is $\T$. The only perfect and regular $A$-function is $D^1$.
\end{theorem}

\begin{proof}
    By the comments following Remark 4.7 of \cite{us2}, no regular $A$-function except $D$ can be structured. %By Theorem 3.10 of that same paper, if $A$ is structured then $\mu_A(p^a)$ will be 1 for some prime power $p^a$ unless $A = D$, excluding all other structured $A$ from being regular.
    
    By Corollary 4.6 in \cite{us2}, $\T$ is the only structured $A$-function $A$ whose $A$-convolution preserves its arithmetical class.
    
    Finally, to see that the only perfect and regular $A$-function is $D^1$, observe that if $A$ is regular and not $D^1$, then there is an $n \in \N$ such that $\mu_A(n) = 0$, but by Corollary 5 of \cite{us4} if $A$ is perfect then $\mu_A(n) \in \{-1,1\}$, so since $D^1$ is the only regular $A$-function with this property, 
\end{proof}

We present the following paradigm for understanding the three families of $A$-functions appearing in the above theorem:

\begin{enumerate}
    \item The regular $A$-functions are those which generalize the properties of $D$ and $D^1$.
    \item The structured $A$-functions are those which generalize the properties of $\T$ and $D$.
    \item The perfect $A$-functions are those which generalize the properties of $D^1$ and $\T$.
\end{enumerate}

In this paper, we will identify a class of $A$-functions to which the regular $A$-functions, the structured $A$-functions, and the perfect $A$-functions belong, but which keeps many relevant properties common to all three of these families. We state these properties below:

\begin{proposition}
    If $A$ is regular, structured, or perfect, then $A$ has the following properties:
    \begin{enumerate}
        \item[\rm(P1)] $A \in C(A)$
        \item[\rm(P2)] When constructing an arbitrary arithmetical function $f \in c(A)$, one may freely choose the value of $f$ at $A$-primitive elements.
        \item[\rm(P3)] Each $n \in \N$ has a unique representation as a product of $A$-primitive elements $n = q_1^{a_1} \cdots q_r^{a_r}$ such that for every $f \in c(A)$, $f(n) = f(q_1)^{a_1} \cdots f(q_r)^{a_r}$.
    \end{enumerate}
\end{proposition}

In the next section, we will introduce a very general class of $A$-functions satisfying (P2) and (P3), and then refine that class to exclude $A$-functions not satisfying (P1) as well.

\section{Complete $A$-functions}\label{complete}

\subsection{Complete $A$-functions and the $A$-decomposition}

For the remainder of this paper, we will be considering $A$-functions that are reflexive. Observe that in this case $\Primes_A$ is a countably-infinite set since $\Primes \subseteq \Primes_A$. We introduce the following concept to aid us in understanding class-$A$ arithmetical functions:

\begin{definition}
    Consider the enumeration of $A$-primitive elements in increasing order (i.e. $q_1 = 2$, $q_2 = 3$, etc.). Let $V(A)$ be the set of all real-valued arithmetical functions $f$ such that there exists $g \in c(A)$ with the property that for all $k \in \N$, $log(g(q_k)) = f(k)$. We call $V(A)$ the \textbf{arithmetical span} of $A$.
\end{definition}

In general we will make use of the above enumeration of a set of $A$-primitive elements unless otherwise mentioned. 

Since $\Ang_{\mathbb{R}}$ may be viewed as $\mathbb{R}^{\infty}$, we may discuss $V(A)$ in terms of its properties relating to this vector space:

\begin{proposition}
    $V(A)$ is a vector subspace of $\mathbb{C}^{\infty}$.
\end{proposition}

\begin{proof}
    Let $A$ be an $A$-function. Since $f(n) \equiv 0$ is in $V(A)$, $V(A)$ is non-empty. Now suppose $f_1$ and $f_2$ are in $V(A)$. Then there is $g_1$ and $g_2$ in $c(A)$ such that $\log(g_i(q_k)) = f_i(k)$. But then $g_1g_2 \in c(A)$ and so $\log(g_1(q_k)g_2(q_k)) = \log(g_1(q_k)) + \log(g_2(q_k)) = f_1(k) + f_2(k)$. Hence, $f_1+f_2 \in V(A)$.
    
    Now suppose $f \in V(A)$ and $c \in \mathbb{R}$. Then if $g$ is the associated arithmetical function in $c(A)$, then $g^c \in c(A)$ as well, so that $\log(g^c(q_k)) = c\log(g(q_k)) = cf(k)$. Hence, $S_A$ is a vector subspace of $\Ang_{\mathbb{R}}$.
\end{proof}

The following defines what will become our most general class of $A$-functions, from which we will refine the aforementioned paradigmatic $A$-functions.

\begin{definition}
    We say that $A$ is \textbf{complete} if $V(A) = \Ang_{\mathbb{R}}$.
\end{definition}

The following weak version of (P2), can be immediately realized:

\begin{proposition}\label{weakP2}
    If $A$ is complete, then when constructing a generic $g \in c(A)$ which is positive and real-valued, we may freely choose the values of $g(q_k)$ independent of all other choices of $g(q_k)$.
\end{proposition}

\begin{theorem}\label{comp}
    The following are equivalent:
    \begin{enumerate}
        \item $A$ is complete.
        \item For every $n \in \N$, $n$ may be uniquely represented as a product of $A$-primitive elements $n = q_1^{a_1} \cdots q_l^{a_l}$ in such a way that for every $f \in c(A)$, $f(n) = f(q_1)^{a_1} \cdots f(q_l)^{a_l}$.
        \item For every $f \in \Ang$ (\textbf{complex}-valued arithmetical functions) there is a $g \in c(A)$ such that for all $k \in \N$, $g(q_k) = f(k)$.
    \end{enumerate}
\end{theorem}

Observe that the second property in Theorem \ref{comp} is (P3). In order to prove this, we need the following Lemma, which essentially guarantees that a weak version of the second property of Theorem \ref{comp} holds for \textit{every} reflexive $A$-function.

\begin{lemma}\label{rep}
    Suppose $A \in \A$ is simple and reflexive. Then for every $n \in \N$ there exists a representation of $n$ as a (not necessarily unique) product of $A$-primitive elements $n = q_1^{a_1} \cdots q_l^{a_l}$ such that for every $f \in c(A)$, $f(n) = f(q_1)^{a_1} \cdots f(q_l)^{a_l}$.
\end{lemma}

\begin{proof}
    We may ``represent'' 1 as an empty product by convention. For $q \in \Primes_A$ we have $f(q) = f(q)$, establishing our base case. Now assuming we have shown this to be true for all $k < n$, we examine $A(n)$. If $n \in \Primes_A$ then we are done. Otherwise, if $d \in A(n)$ with $1 < d < n$ then for every $f \in c(A)$, $f(d) f(\frac{n}{d}) = f(n)$. But then if $f(d) = f(q_1)^{b_1} \cdots f(q_l)^{b_l}$ and $f(\frac{n}{d}) = f(q_1)^{c_1} \cdots f(q_l)^{c_l}$ then $f(n) = f(q_1)^{a_1} \cdots f(q_l)^{a_l}$ for $a_i = b_i + c_i$, so by induction we are done.
\end{proof}

Now we may prove our first main result.

\begin{proof}
    Suppose $A$ is complete. We will show that the representation of each $n \in \N$ guaranteed by Lemma \ref{rep} is unique. Let $n \in \N$ be arbitrary and suppose $n = q_1^{a_1} \cdots q_l^{a_l} = q_1^{b_1} \cdots q_l^{b_l}$ are two representations of $n$, where we allow some $a_i$ and $b_j$ to be zero if needed in order to maintain identical enumerations on the left and right hand side. Then by Lemma \ref{rep}, for every $g \in c(A)$, $g(q_1)^{a_1} \cdots g(q_l)^{a_l} = g(q_1)^{b_1} \cdots g(q_l)^{b_l}$. Consider the $g \in c(A)$ which are positive real-valued. Then taking logarithms of the above equation gives us 
    \begin{equation}\label{eq:logdep} a_1\log(g(q_1)) + \cdots + a_l\log(g(q_l)) = b_1\log(g(q_1)) + \cdots + b_l\log(g(q_l)),\end{equation} which holds for all positive real-valued $g \in c(A)$.
    
    We may write 
    \begin{equation}
        (a_1-b_1)log(g(q_1)) + \dots + (a_l-b_l)\log(g(q_l)) = 0.
    \end{equation} 
    
    As we range over all positive real-valued $g$, the vectors whose $i^{th}$ components are the values $(a_i-b_i)log(g(q_i))$ span a subspace $\mathbb{R}^l$. If this subspace is non-trivial then there is a $k$ such that $(a_k - b_k) \neq 0$, allowing us to express $log(g(q_k))$ in terms of the other $log(g(q_i))$, $i \neq k$. However, this contradicts Proposition \ref{weakP2}, so $a_i = b_i$ for all $i$ and the representation is unique. 
    
    %implies (presuming without loss of generality that $a_1 \neq 0$) that $\log(g(q_1))$ depends on the other $\log(g(q_i))$ and $\log(g(r_j))$ unless the left and right expressions are algebraically identical. 

    %Since $A$ is complete, we cannot have any equations that imply $g(q_1)$ depends on any of the other $g(q_k)$ (as this would result in a non-trivial null space for $V(A)$) and so the left and right expressions must be identical, which is to say that the representations of $n$ are identical, and hence unique.
    
    Now assume $A$ is such that we have a unique representation of every $n \in \N$ in the manner described above. Let $f \in \Ang = \mathbb{C}^{\infty}$ be given. Define $g \in \Ang$ on $A$-primitive elements $q_k$ to be $g(q_k) = f(k)$ for all $k \in \N$. Then for $n = q_1^{a_1} \cdots q_l^{a_l}$ the unique representation of $n$, let $g(n) = \prod_{k = 1}^l  g(q_k)^{a_k}$. By definition, this function is in $c(A)$, implying that we may find such a $g$ for every $f \in \Ang$. %We show $g \in c(A)$. If $n$ is $A$-primitive then there is nothing to show. Otherwise, if $d \in A(n)$ is such that $d \neq 1,n$ then observe that, for $d = q_1^{b_1} \cdots q_l^{c_l}$ and $\frac{n}{d} = q_1^{c_1} \cdots q_l^{c_l}$ we have $n = q_1^{a_1} \cdots q_l^{a_l}$ and so $g(d)g(\frac{n}{d}) = g(n)$. Hence, $g \in c(A)$.
    
    Finally, suppose that $A$ is such that for every $f \in \Ang$ there is $g \in c(A)$ such that $g(q_k) = f(k)$ for all $k \in \N$. Then in particular, this holds for all positive real-valued $g$, and hence $V(A) = \Ang_{\mathbb{R}}$ so $A$ is complete.
\end{proof}

The unique representation of $n$ demonstrated to exist for complete $A$ will be called the \textbf{$A$-decomposition of $n$}. The following are immediate consequences of Theorem \ref{comp}:

\begin{corollary}
    If $A$ is complete, then $A$ satisfies \rm{(P2)} and \rm{(P3)}.
\end{corollary}

\begin{corollary}
    If $A$ is regular, structured, or perfect, then $A$ is complete.
\end{corollary}

Let us now give some examples of complete $A$-functions, as well as examples of $A$-functions which are not complete.

\begin{example}
    The $A$-functions $D$, $D^1$, and $\T$ are all complete. Every integer $n$ is naturally decomposed as \[\begin{aligned}n &= (p_1)^{a_1} \cdots (p_l)^{a_l} & \quad (p_i \in \Primes_D = \Primes) \\ &= (p_1^{a_1}) \cdots (p_l^{a_l}) & \quad (p_i^{a_i} \in \Primes_{D^1}) \\ &= (p_1^{2^0})^{b_{1,0}} (p_1^{2^1})^{b_{1,1}} \cdots (p_l^{2^0})^{b_{l,0}} (p_1^{2^1})^{b_{l,1}}) \cdots & \quad (p_i^{2^j} \in \Primes_{\T}) \end{aligned},\] and it is well-known that \[f \in c(D) \Rightarrow f(n) = f(p_1)^{a_1} \cdots f(p_l)^{a_l}\] and \[f \in c(D^1) \Rightarrow f(n) = f(p_1^{a_1}) \cdots f(p_l)^{a_l}.\] By Corollary 4 of \cite{us4}, \[f \in c(\T) \Rightarrow f(n) = f(p_1^{2^0})^{b_{1,0}}f(p_1^{2^1})^{b_{1,1}} \cdots f(p_l^{2^0})^{b_{l,0}}f(p_1^{2^1})^{b_{l,1}}),\] where we are able to find $f \in c(A)$ in each case to show that this is unique (for example, $f = I$ for $D$ and $f = \varphi$ for $D^1$). In Section \ref{maxF} we will deduce for any complete $A$-function $A$ that a single function can be found whose ``largest'' arithmetical class is $c(A)$.
\end{example}

\begin{example}
    Any cross-convolution of complete $A$-functions is also complete.
\end{example}

\begin{example}
    For an $A$-function which is not complete, consider the homogeneous $A$-function $A$ defined on prime powers as \[A(p^a) = 
    \begin{cases}
        \{1,p,p^2,p^3,p^4\} & a = 4 \\
        D^1(p^a) & \text{\rm otherwise}
    \end{cases} \]
    Notice that every $f \in c(A)$ must satisfy $f(p^4) = f(p^2)^2 = f(p)f(p^3)$ (since $p^2 \in A(p^4)$ and $p \in A(p^4)$) for every prime $p$, so in particular, even though $p,p^2,p^3 \in \Primes_A$, choosing $f(p)$ and $f(p^3)$ restricts choice of $f(p^2)$. As we will see in Section \ref{maxF}, this excludes $A$ from being regular, structured, or perfect.
    
    However, note that $A$ is associative, indicating that completeness is not a sufficiently general property to capture all associative $A$-functions.
\end{example}

\begin{example}\label{exG}
    To see a non-symmetric complete $A$-function, consider the homoegeneous $A$-function $G$ introduced in \cite{us3}, which is the $A$-function satisfying \[d \in G(n) \Leftrightarrow D(d) \cap \T\parens{\frac{n}{d}} = \1.\] By examining the Cohen Triangle of $G$, one can deduce that
    \begin{enumerate}
        \item The $G$-convolution $*_G$ is associative.
        \item The $G$-decomposition of any integer $n$ is equal the $\T$-decomposition. 
    \end{enumerate}
    Neither of these assertions are immediate. Nevertheless, we leave it to the reader to check these facts.
\end{example}

\subsection{Paradigmatic $A$-functions}

Note that any complete $A$-function $A$ induces an $A$-decomposition of every $n \in \N$, but the $A$-decomposition does not necessarily uniquely determine $A$ (cf. Examaple \ref{exG}).

\begin{definition}
    Suppose $A$ is complete. Then for all $n \in \N$, expressed via the $A$-decomposition $n = q_1^{a_1} \cdots q_l^{a_l}$, $q_k \in \Primes_A$, suppose $A(n) = A(q_1)^{a_1} \cdots A(q_l)^{a_l}$. In such a case we call $A$ \textbf{\Aname}.
\end{definition}

It follows via an inductive argument that \Aname $A$-functions are complete, which relies on the use of a complete $A$-function in their construction. 

We immediately obtain the following:

\begin{proposition}\label{RSPfrac}
    If $A$ is regular, structured, or perfect, then $A$ is \Aname.
\end{proposition}

\begin{proof}
    In each case, since $A$ is complete we may decompose each $n \in \N$ via the $A$-decomposition as $n = q_1^{a_1} \cdots q_l^{a_l}$. Note that $A(q_i) = \{1,q_i\}$ since $q_i$ is $A$-primitive. Then regardless of whether $A$ is regular, stuctured, or perfect, $q_1^{b_1} \cdots q_l^{b_l} \in A(n)$ for all $0 \leq b_i \leq a_i$ for all $i$, and $d \in A(n)$ if and only if it is of this form. But the set of all such $d$ may also be written as $A(q_1)^{a_1} \cdots A(q_l)^{a_l}$, so $A$ is \Aname.
\end{proof}

We then have the following main result:

\begin{theorem}\label{frac}
    $A$ is \Aname if and only if $A \in C(A)$.
\end{theorem}

\begin{proof}
    First, notice that by definition, if $A$ is \Aname and $d \in A(n)$, then we may write $n = q_1^{a_1} \cdots q_l^{a_l}$, $d = q_1^{d_1} \cdots q_l^{d_l}$, and $\frac{n}{d} = q_1^{c_1} \cdots q_l^{c_l}$, where $a_k = d_k + c_k$. But then, $A(d) = A(q_1)^{d_1} \cdots A(q_l)^{d_l}$ and $A(\frac{n}{d}) = A(q_1)^{c_1} \cdots A(q_l)^{c_l}$, so we indeed have $A(d) A(\frac{n}{d}) = A(n)$, so $A \in C(A)$.
    
    On the other hand, if $A \in C(A)$ then we proceed by induction on $n$. For any $A$-primitive element $q$, the unique $A$-decomposition is trivially $q = q$, giving us $A(q) = A(q)$. If for all $k < n$, $k = q_1^{k_1} \cdots q_l^{k_l}$, we have deduced a unique $A$-decomposition giving us $A(k) = A(q_1)^{k_1} \cdots A(q_l)^{k_l}$ for all $k < n$, then examine $A(n)$. If $n \in \Primes_A$ then we are done, so assume there is $d \in A(n)$ such that $1 < d < n$. Then $A(d) A(\frac{n}{d}) = A(n)$. If we let $d = q_1^{d_1} \cdots q_l^{d_l}$, $\frac{n}{d} = q_1^{c_1} \cdots q_l^{c_l}$, and define $a_i = d_i + c_i$, then $A(n) = A(d)A(\frac{n}{d}) = A(q_1)^{a_1} \cdots A(q_l)^{a_l}$.
    
    To complete the induction step we claim that $n = q_1^{a_1} \cdots q_l^{a_l}$ is \textit{the} $A$-decomposition of $n$. Suppose $n = q_1^{a_1} \cdots q_l^{a_l} = q_1^{b_1} \cdots q_l^{b_l}$. Let $d$ be the least element of $A(n)$ larger than one, so that $d \in \Primes_A$. Hence, the first $a_i$ and the first $b_i$ that are not zero must occur at the same $i$, say $i_0$. Then observe that $q_{i_0} \cdot q_{i_0}^{a_{i_0}-1} \cdot q_{i_0+1}^{a_{i_0+1}} \cdots q_l^{a_l} = q_{i_0} \cdot q_{i_0}^{b_{i_0}-1} \cdot q_{i_0+1}^{b_{i_0+1}} \cdots q_l^{b_l}$. But then since the $A$-decomposition is unique for all $k < n$, $a_i = b_i$ for all $i$. Hence, we are finished by induction.
\end{proof}

By Proposition \ref{RSPfrac}, the following is immediate.

\begin{corollary}
    If $A$ is regular, structured, or perfect, then $A \in C(A)$.
\end{corollary}

\begin{corollary}
    \rm{(P1)} \textit{implies} \rm{(P2)} \textit{and} \rm{(P3)}.
\end{corollary}

\begin{proof}
    $A \in C(A)$ implies $A$ is \Aname, so $A$ is complete, which implies \rm{(P2)} and \rm{(P3)}.
\end{proof}

\begin{corollary}
    If $A$ is \Aname then $A$ is split.
\end{corollary}

\begin{proof}
    Let $A$ be \Aname and let $m \in A(mn)$, $a \in A(m)$, and $b \in A(n)$. Then $ab \in A(mn)$, and since $A \in C(A)$ we have that $A(m) A(n) = A(a)A(b)A(\frac{m}{a})A(\frac{n}{b}) = A(mn)$. Hence, $ab$ has as its $A$-decomposition the product of the $A$-decomposition of $a$ and $b$, and so since $A$ is fractured, $a \in A(ab)$.
\end{proof}

Of course, this means that we may refine Lemma 5 of \cite{us4}:

\begin{theorem}
    Suppose $A \in C(A)$. Then $*_A$ is associative.
\end{theorem}

The only example appearing in the literature of a non-\Aname $A$-function for which $*_A$ is associative is the $A$-function $G$ from \cite{us3} (cf. Example \ref{exG} above), which demonstrates also that $A$ need not be symmetric in order for $*_A$ to be associative, a fact previously unstated.

We summarize the preliminary results on \Aname $A$-functions in the following:

\begin{theorem}
    If $A$ is \Aname, then:
    \begin{enumerate}
        \item $A \in C(A)$
        \item $A$ is complete.
        \item $*_A$ is associative
        \item $A$ is symmetric
    \end{enumerate}
\end{theorem}

Since all regular, structured, or perfect $A$-functions are \Aname, we present examples of $A$-functions that belong to none of these classes, yet are \Aname.

\begin{example}
    Any cross-convolution of \Aname $A$-functions is \Aname. In particular, suppose $A$ is the cross-convolution of $D$, $D^1$, and $\T$, where, say, $A(2^a) = D(2^a)$, $A(p^a) = D^1(p^a)$ if $p \equiv 1$ (mod 4), and $A(p^a) = \T(p^a)$ if $p \equiv 3$ (mod 4). Then $A$ belongs to none of the three families.
\end{example}

\begin{example}
    Let $A$ be the $A$-function which for all $n \in \N$ not a power of 2 is $A(n) = Z(n) = \{1,n\}$, but if $n = 2^a$ is $A(2^a) = D(2^a)$ for all $a \in \N$. Note that $A$ is obviously not multiplicative in this case.
\end{example}

\begin{example}
    To see an example of a homogeneous $A$-function that is not regular, structured, or perfect and yet is \Aname, consider the $A$-function $A$ defined on prime powers via 
    \[A(p^a) = 
    \begin{cases}
        \{1,p^2,p^4\} & a = 4 \\
        \{1,p^2,p^3,p^5\} & a = 5 \\
        \{1,p^2,p^3,p^4,p^5,p^7\} & a = 7 \\
        D^1(p^a) & \text{\rm otherwise}
    \end{cases} \]
    Then this $A$-function is \Aname.
\end{example}

\subsection{The maximal class of an arithmetical function}\label{maxF}

One can notice that, when speaking of arithmetical classes for a given arithmetical function $f$, a natural ordering emerges in which those functions $f \in c(Z)$ (i.e., $f(1) = 1$) constitute a very general class of arithmetical function, whereas the functions $f \in c(D)$ are in a sense the least general class. Indeed, $f \in c(D)$ implies $f \in c(A)$ for all $A$, and in general if $A \subseteq B$ then $c(B) \subseteq c(A)$. We wish to switch the focus from finding an $A$ for which $f \in c(A)$ to finding the maximal $A$ for which this is true.

\begin{definition}
    Let $f \in c(Z)$. Denote by $m(f) := A$ the $A$-function such that $m \in A(mn)$ if and only if $f(mn) = f(m)f(n)$. We call this the \textbf{maximal class of $f$}.
\end{definition}

It is clear from the definition that for any $f \in \Ang$ with $f(1) = 1$, $m(f)$ is well-defined and if $A$ is any $A$-function with $f \in c(A)$, then $A \subseteq m(f)$.

\begin{example}
    Consider the Euler $\varphi$ function. A well-known identity (cf. \cite{apostol}) involving $\varphi$ is \[\varphi(mn) = \varphi(m)\varphi(n) \frac{(m,n)}{\varphi(m,n)}.\] We can immediately see that unless $(m,n) = 1$, then $\varphi(mn) \neq \varphi(m)\varphi(n)$, but of course $(m,n) = 1 \Leftrightarrow m \in D^1(mn)$, so $m(\varphi) = D^1$. Functions of this type were examined in a very general setting by \cite{haukR}
\end{example}

\begin{example}
    Consider the divisor function $\tau$. Then for $m = p_1^{a_1} \cdots p_l^{a_l}$ and $n = p_1^{b_1} \cdots p_l^{b_l}$, $\tau(m)\tau(n) = (a_1 + 1) \cdots (a_l + 1) (b_1 + 1) \cdots (b_l + 1)$, which must be strictly larger than $(a_1+b_1+1) \cdots (a_l+b_l+1) = \tau(mn)$ unless $(m,n) = 1$ (i.e., $a_i > 0 \Rightarrow b_i = 0$), which means $m(\tau) = D^1$ as well.
\end{example}

\begin{example}
    Consider the \Mobius function $\mu$. Then the equation $\mu(mn) = \mu(m)\mu(n)$ can be split into two cases, according to whether or not the left and right hand sides are equal to zero. If they are not equal to zero, then we need $(m,n) = 1$ in addition to $m$ and $n$ being individually square-free.
    
    If $\mu(mn) = 0$, then since $\mu(m)\mu(n) = 0$, one of $m$ or $n$ must not be square-free, and for all such $m$ and $n$, $\mu(mn) = \mu(m)\mu(n)$. So $m(\mu)$ is the $A$-function $A$, where $m \in A(mn)$ if either $m \in D^1(mn)$ for $m$ and $n$ square-free or $m$ or $n$ is not square-free.
    
    Note that $m(\mu)$ is not multiplicative even though $\mu$ is, and in general it can be true that $m(f)$ appears to ``lose'' properties analogous to those that $f$ possesses.
\end{example}

\begin{example}
    Initially, one might think that $m(f)$ ought to be complete for any $f$. Unfortunately, this is not the case. Define 
    \[f(n) := 
    \begin{cases}
        \exp(\pi^n) & n \neq 2,4,8,16\\
        2i & n = 2 \\
        4 & n = 4 \\
        -8i & n = 8 \\
        16 & n = 16.
    \end{cases}
    \]
    It is then clear that
    \[m(f)(n) = 
        \begin{cases}
            Z(n) & n \neq 16 \\
            \{1,2,4,8,16\} & n = 16,
        \end{cases}
    \]
    and yet $m(f)$ is not complete.
\end{example}

\begin{example}
    On the other hand, even if $A$ is complete, it does not mean that $f \in c(A)$ can be found such that $m(f) = A$. Indeed, any complete $A$-function which is not symmetric will lack this property, as $m(f)$ is symmetric by definition.
\end{example}

We do, however, have the following:

\begin{theorem}\label{mf}
    If $A$ is \Aname, then there exists a function $f \in c(A)$ with $m(f) = A$. 
\end{theorem}

\begin{proof}
    Let $A$ be \Aname and define for any real number $x$ the function $f_x \in c(A)$ on $A$-primitive elements to be $f_x(q_k) = \exp(x^k)$. For a given expressed via its $A$-decomposition as $n = q_1^{c_1} \cdots q_l^{c_l}$, it is clear that $f_x(n)$ is the exponential of a polynomial in $x$ for all $x$, dependent only on $n$. Call this polynomial $P_n$. Then $P_n(x) = c_1x^1 + c_2x^2 + \cdots c_lx^l$.
    
    Let $m = q_1^{b_1} \cdots q_l^{b_l}$ and let $mn = q_1^{a_1} \cdots q_l^{a_l}$ be the $A$-decomposition of $m$ and $mn$, respectively. Suppose that $f_x(mn) = f_x(m)f_x(n)$ for all $x \in \mathbb{R}$. Then $log(f_x(mn)) = log(f_x(m)) + log(f_x(n))$, and hence $P_{mn} = P_m + P_n$ (as non-zero polynomials may only possess a finite number of solutions). However, this implies that $a_i = b_i + c_i$ for all $i$, which is to say that the $A$-decomposition of $mn$ is the product of the $A$-decompositions of $m$ and $n$. 
    
    Finally, since $A$ is \Aname it follows that $m \in A(mn)$, as each $b_i$ satisfies $0 \leq b_i \leq a_i$ for all $i$. Choose $x$ to be a non-algebraic real number. Then for all $m,n \in \N$, $P_{mn} = P_m + P_n$ if and only if $P_{mn}(x) = P_m(x) + P_n(x)$, as all such polynomials have coefficients in $\NatZero$. Thus, $m(f_x) = A$.
\end{proof}

An analogous version of our $m(f)$ has appeared in the literature before, in the work of Ryden \cite{ryden}. In that work, for $f \in \Ang$ the set $T(f)$ is defined to be the set of all ordered pairs $(a,b)$ such that $f(ab) = f(a)f(b)$. Clearly, $(a,b) \in T(f)$ is equivalent to $a \in m(f)(ab)$ for $f(1) = 1$.

The scope of Ryden's work was confined to investigating groups of the form $(c(A),*)$, that is, groups formed from arithmetical classes under Dirichlet convolution. We may reformulate and prove his main result using results from \cite{us4}, but first need the following definition from that work.

\begin{definition}[Definition 10 from \cite{us4}]
    Let $f \in \Ang$ be an arithmetical function with $f(1) = 1$. Let $A$ and $B$ be symmetric $A$-functions such that $m \in B(mn)$, $a \in A(m)$, and $b \in A(n)$ together imply that $f(ab) = f(a)f(b)$. We then call $f$ \textbf{$(A,B)$-split}. If under the same premises, $a \in H(ab)$ for an $A$-function $H$, then we say that $H$ is $(A,B)$-split.
\end{definition}

\begin{theorem}[Restatement of Theorem 3.8 of Ryden]
    Suppose $A \subseteq D^1$ is a symmetric, reflexive $A$-function. Then $(c(A),*)$ is a group if and only if for all $f \in c(A)$, $f$ is $(D,A)$-split.
\end{theorem}

\begin{proof}
    By Theorem 1 of \cite{us4}, $*_D = *$ preserves $c(A)$ if and only if
    \begin{enumerate}
        \item $D \in C(A)$
        \item $\tau_D = \tau \in c(A)$
        \item $\forall f \in c(A)$, $f$ is $(D,A)$-split.
    \end{enumerate}
    
    The third condition is explicitly satisfied by assumption. Since $D \in C(D)$, $D \in C(A)$ for all $A \in \A$. Finally, $\tau \in c(D^1)$ and since $A \subseteq D^1$, $\tau \in c(A)$. Then by Corollary 3 of \cite{us4}, since $*$ is associative, $(c(A), *)$ forms a group.
    
    Going the other way, we see that if $(c(A), *)$ forms a group then $*$ must preserve $c(A)$, so in particular since $m(\tau) = D^1$, we need $A \subseteq D^1$ and we need all $f \in c(A)$ to be $(D,A)$-split, so we are done.
\end{proof}

\subsection{Arithmetical functions associated with complete $A$-functions}

\begin{definition}
    Let $A$ be complete and let $n = q_1^{a_1} \cdots q_l^{a_l}$ be the $A$-decomposition of $n$. Define $\omega_A(n) := l$ and define, for $q \in \Primes_A$, $\nu_{A,q}(n)$ to be $a_i$ if $q = q_i$ for some $i$, and 0 otherwise.
\end{definition}

\begin{definition}
    Let $A$ be complete. Let $\prim_A(n)$ be the set of $A$-primitive elements in the $A$-decomposition of $n$.
\end{definition}

It is easy to see $\omega_A(n) = |\prim_A(n)|$. We also have the following:

\begin{proposition}\label{PA}
    Suppose $A$ is complete and let $B := \perf(A)$ be the $A$-function defined at each $n \in \N$ to be, for $n = q_1^{a_1} \cdots q_l^{a_l}$ the $A$-decomposition of $n$,
    \[B(n) = 
    \begin{cases}
        \{1,n\} & \text{ if }\omega_A(n) = 1\\
        B(q_1^{a_1}) \cdots B(q_l^{a_l}) & \text{otherwise.}
    \end{cases} \]
    Then $B$ is perfect.
\end{proposition}

\begin{proof}
    The fact that $B = \perf(A)$ is such that $B \in C(B)$ follows from the definition. We show that $B \subseteq B^1$. If $n \in \Primes_B$ then $B(n) = B^1(n)$, so assume $n \notin \Primes_B$. We need to show that $d \in B(n) \Rightarrow B(d) \cap B(\frac{n}{d}) = \1$. If $d = q_1^{b_1} \cdots q_l^{b_l}$ is the $A$-decomposition of $d$ and $\frac{n}{d} = q_1^{c_1} \cdots q_l^{c_l}$ is the $A$-decomposition of $\frac{n}{d}$, then for each $i$, $b_i$ or $c_i$ is 0, for the elements of $B(n)$ are of the form $q_{i_1}^{a_{i_1}} \cdots q_{i_k}^{a_{i_k}}$. Hence, $B(d) \cap B\parens{\frac{n}{d}} = \1$ and we are done.
\end{proof}

Hence, for example, $\perf(D) = D^1$ and $\perf(G) = \T$. For any structured $A$-function $B$, $\perf(B) = B^1$. We also have the following:

\begin{proposition}\label{FA}
    Suppose $A$ is complete and let $F := F(A)$ be the $A$-function defined at each $n \in \N$ to be, for $n = q_1^{a_1} \cdots q_l^{a_l}$,
    \[F(n) = A(q_1)^{a_1} \cdots A(q_l)^{a_l}\]
    Then $F$ is \Aname.
\end{proposition}

\begin{proof}
    This $A$-function is well-defined due to the unique $A$-decomposition of each $n \in \N$, and so by construction is \Aname.
\end{proof}

Propositions \ref{PA} and \ref{FA} tell us that each complete $A$-function may be associated with a perfect $A$-function and a \Aname $A$-function.

\begin{lemma}\label{perfprim}
    Suppose $A$ is complete and suppose $f \in \Ang$ is such that $f(1) \neq 0$ and $f(mn) = f(m)f(n)$ whenever $\prim_A(m) \cap \prim_A(n) = \emptyset$. Then $f \in c(\perf(A))$.
\end{lemma}

\begin{proof}
    Let $B = \perf(A)$ and suppose $m \in B(mn)$. Then since $B$ is perfect, $m \in B(mn) \Rightarrow m \in B^1(mn) \Rightarrow B(m) \cap B(n) = \1 \Rightarrow \prim_B(m) \cap \prim_B(n) = \emptyset$ since $\prim_B(m)$ and $\prim_B(n)$ are subsets of $B(m)$ and $B(n)$, respectively, which do not contain 1. But then our result follows.
\end{proof}

For the purposes of the next result, let $A$ be complete. We say that $n \in \N$ is \textbf{$A$-square-free} if in the $A$-decomposition of $n = q_1^{a_1} \cdots q_l^{a_l}$, each $a_i$ is either 0 or 1.

\begin{theorem}
    Let $A$ be complete and let $B := \perf(A)$ and $F := F(A)$. Then $\mu_F,\mu_B \in c(B)$, $\omega_A = \omega_F = \omega_B$, $\mu_B(n) = (-1)^{\omega_A(n)}$, and furthermore
    \begin{equation}\label{eq:muF}
    \mu_F(n) = 
        \begin{cases}
            \mu_B(n) & n\text{is A-square-free}\\
            0 & \text{otherwise.}
        \end{cases}
    \end{equation}
\end{theorem}

\begin{proof}
    The fact that $\omega_A = \omega_F = \omega_B$ is immediate from the definition of $\omega_A$ and $\perf(A)$.
    
    Suppose $m,n \in \N$ are such that $\prim_A(m) \cap \prim_A(n) = \emptyset$. If $m \in B(mn)$, then $\prim_A(mn) = \prim_A(m) \cup \prim_A(n)$ and so $\omega_A(mn) = \omega_A(m) + \omega_A(n)$. Then $(-1)^{\omega_A(mn)} = (-1)^{\omega_A(m)}(-1)^{\omega_A(n)}$, so $g(n) := (-1)^{\omega_A(n)}$ is class-$B$.
    
    To see that $g = \mu_B$, we may check that for $q \in \Primes_B$, $1 + g(q) = 1-1 = 0$, since $g \in c(B)$, so $g = \mu_B$ and we are done. To see that \eqref{eq:muF} gives the correct expression for $\mu_F$, first apply the same reasoning as for $\mu_B$ to see that $\mu_F \in c(B)$. Then we may check, for $q \in \Primes_A = \Primes_{F(A)}$, $(u *_F \mu_F)(q^a) = 1 + \mu_F(q) = 1 - 1 = 0$ for all $a$ such that $q^a$ is the $A$-decompositon of $n = q^a$.
\end{proof}

\subsection{Classification of complete $A$-functions}

The next few results enable us to understand the relationship between complete $A$-functions, the subclass of \Aname $A$-functions and the further subclass of perfect $A$-functions.

\begin{lemma}\label{minB}
    Suppose $A$ is complete and $B$ is such that
    \begin{enumerate}
        \item $\Primes_B = \Primes_A$
        \item For each $n \in \N$, $n \notin \Primes_B$, $n > 1$, there is some $d \in A(n)$, $1 < d < n$, such that $B(n) = \{1,d,n\}$
    \end{enumerate}
    Then $B$ is complete and $c(B) = c(A)$.
\end{lemma}

In general, if $B$ is a complete $A$-function such that $\tau_B(n) \leq 3$ for all $n$, we call $B$ \textbf{minimally complete}.

\begin{proof}
    Observe that by the proof of Lemma \ref{rep}, we may decompose any $n \in \N$ using $B$ in the same manner as $A$, since $B \subseteq A$. However, $B$ is such that each $n \in \N$ has only a single ``path'' along which it can be decomposed. Namely, if $n > 1$ is not $B$-primitive, then there is only one $d \in B(n)$ such that $1 < d < n$, enabling one to inductively show, starting from the $B$-primitive elements (which are equal to the $A$-primitive elements by assumption), that $B$-decomposition of $n$ is unique and equal to the $A$-decomposition of $n$ for all $n$.
\end{proof}

\begin{theorem}\label{compF}
    If $A$ is complete, then $c(A) = c(F(A))$.
\end{theorem}

\begin{proof}
    It is obvious that $\Primes_A = \Primes_{F(A)}$ by the definition of $F(A)$. For each $n \notin \Primes_A$, choose $d \in A(n)$ with $1 < d < n$ and construct a minimally complete $A$-function $B$ via $B(n) = \{1,d,n\}$ for these $n$, letting the $B$-primitive elements be equal to the $A$-primitive elements. Then the conditions of Lemma \ref{minB} hold for $A$ and $B$, as well as for $F(A)$ and $B$, giving us $c(A) = c(B) = c(F(A))$.
\end{proof}

\begin{theorem}\label{indComp}
    Suppose $A$ is complete and $B$ is such that $c(A) = c(B)$. Then $B$ is complete.
\end{theorem}

\begin{proof}
    Let $A$ be complete and let $B$ be such that $c(A) = c(B)$. We argue that $B \subseteq F(A)$. Indeed, were this not the case there would be some minimal $n \in \N$ such that $d \in B(n)$ but $d \notin F(A)(n)$. Hence $f \in c(B)$ would mean $f(d)f(\frac{n}{d}) = f(n)$, a relation not satisfied by \textbf{every} $f \in c(F(A))$ since there is such an $f$ with $m(f) = F(A)$ by Theorem \ref{mf}. Thus there can be no such value and $B \subseteq F(A)$.
    
    We can easily show that $\Primes_B = \Primes_{F(A)} = \Primes_A$; else, there would be relations satisfied by all $f \in c(B)$ that are not satisfied by all $f \in c(A)$, contradicting the assumption. This being said, if $n \notin \Primes_B$ then there is some proper non-trivial $m \in F(A)(n)$ such that $m \in B(n)$, and this would be true for each such $n$. Hence there is some minimally complete $\tilde{B}$ contained in $B$ such that $\Primes_{\tilde{B}} = \Primes_B$. Indeed, $F(B) = F(A)$.
\end{proof}

\begin{corollary}\label{corrComp}
    Suppose $A$ is complete, $B$ is minimally complete with $c(A) = c(B)$, and $B'$ is such that $B \subseteq B' \subseteq F(A)$. Then $B'$ is complete and $c(B') = c(A)$.
\end{corollary}

\begin{proof}
    Since $B \subseteq B' \subseteq F(A)$, $c(F(A)) \subseteq c(B') \subseteq c(B)$. However, $c(B) = c(F(A))$ by Theorem \ref{compF}, so $c(B') = c(F(A)) = c(A)$, and so by Theorem \ref{indComp} $B'$ is complete.
\end{proof}

We may now easily see that $F(G) = \perf(G) = \T$ since $c(G) = c(\T)$.

\begin{example}
    Let $A$ be complete and let $\Gamma(A)$ be the reflexive $A$-function such that $\Gamma(A)(n) = \{1,n\} \cup P_A(n)$. Then $\Gamma$ is complete by Corollary \ref{corrComp} and $c(\Gamma(A)) = c(A)$.
\end{example}

The following will be needed in the next section:

\begin{lemma}\label{ABsplit}
    If $A$ is complete, $F:=F(A)$, and $B := \perf(A)$, then $F$ is $(F,P)$-split.
\end{lemma}

\begin{proof}
    Suppose $m \in B(mn)$, $a \in F(m)$, and $b \in F(n)$. Then since $B \subseteq F$, $m \in F(mn)$, and since $F \in C(F)$, $ab \in F(mn)$, so $F(ab) \subseteq F(mn)$. An argument involving the $F$-decomposition of $ab$ reveals that we must have $a \in F(ab)$, so we are done.
\end{proof}

\section{Results for $A$-rafts and examples}\label{Results}

\subsection{Main result for $A$-rafts}

We have the following:

\begin{theorem}\label{close}
    Suppose $A$ is complete. Then the closure of $\Rat(F(A))$ under the above metric is $c(P(A))$.
\end{theorem}

\begin{proof}
    Assume $A$ is complete and let $F = F(A)$ and let $B = \perf(A)$. First we show that $*_F$ preserves $c(B)$ since this will show that $\Rat(F) \subseteq c(B)$. By Theorem 1 of \cite{us4}, $*_F$ preserves $c(B)$ if and only if \begin{enumerate}
        \item $F \in C(B)$
        \item $\tau_F \in c(B)$
        \item For every $f \in c(B)$, $f$ is $(F,B)$-split.
    \end{enumerate}
    
    By Theorem \ref{frac}, $F \in C(F)$ and hence $F \in C(B)$ since $B \subseteq F$. Furthermore, Lemma \ref{perfprim} combined with the definition of $F$ implies that $\tau_F \in c(B)$. Finally, let $f \in c(B)$. If $m \in B(mn)$ $a \in F(m)$ and $b \in F(n)$ then the following is true: Since $m \in B(mn)$, we may relabel the $A$-primitive elements $q_i$ and write this generally as $q_1^{a_1} \cdots q_l^{a_l} \in B(q_1^{a_1} \cdots q_l^{a_l}q_{l+1}^{a_{l+1}} \cdots q_k^{a_k})$, and we may write $a$ as $q_1^{c_1} \cdots q_l^{c_l}$ and $b$ as $q_{l+1}^{c_{l+1}} \cdots q_k^{c_k}$. 
    
    But since $F$ is $(F,B)$-split by Lemma \ref{ABsplit}, $a \in F(ab)$ and we have $q_1^{c_1} \cdots q_l^{c_l} \in B(q_1^{c_1} \cdots q_l^{c_l} \cdots q_{l+1}^{a_{l+1}} \cdots q_k^{c_k})$ which itself implies $B$ is $(F,B)$-split but in particular implies $f(ab) = f(a)f(b)$. So indeed, $*_A$ preserves $c(B)$. This implies that the closure of $\Rat(A)$ is a subset of $c(B)$, as $c(B)$ is closed under the $A$-convolution and hence all $A$-convolutions of arithmetical functions are in $c(B)$.
    
    We demonstrate that the closure of $\Rat(A)$ is $c(B)$. First, note that $c(B)$ is closed, since every sequence $\{f_n\}$ of functions in $c(B)$ will consist entirely of functions which are class-$B$, and hence cannot converge to a function lacking this property. 
    
    Now take $f \in c(B)$. We demonstrate that this $f$ may be approximated by a sequence of rational functions from $\Rat(A)$. Since we may freely choose $A$-primitive elements for some class-$A$ arithmetical function $g_1$, we may choose these in such a way that $g_1(n) = f(n)$ if $n \in \Primes_A$. Since $f \in c(B)$, it is sufficient to construct a sequence $g_k$ in $\Rat(F)$ such that in the limit $g_k(n) \rightarrow f(n)$ for all $n \in \Primes_B$. 
    
    Let $h_1 = g_1$. Then suppose $n$ has $A$-decomposition $q^2$. We will show that we can find class-$A$ functions $g_{2,1}$ and $g_{2,2}$ such that $(g_{2,1}*_Ag_{2,2})(q) = f(q)$ and $(g_{2,1}*_Ag_{2,2})(q^2) = f(q^2)$. Notice that this turns into the following system of equations:
    
    \[\begin{aligned} f(q) &= g_{2,1}(q) + g_{2,2}(q)   \\ f(q^2) &=  g_{2,1}(q^2) + g_{2,1}(q)g_{2,2}(q) + g_{2,2}(q^2).\end{aligned}\] 
    
    We may express $g_{2,1}(q)$ in terms of $g_{2,2}(q)$ as $g_{2,1}(q) =- g_{2,2}(q) +f(q)$, which turns our second equation into \[\begin{aligned}f(q^2) &= - g_{2,2}(q) +f(q))^2 + (- g_{2,2}(q) +f(q))g_{2,2}(q) + g_{2,2}(q^2)  \\ &=f(q)^2 - 2f(q)g_{2,2}(q) + g_{2,2}(q)^2 + - g_{2,2}(q)^2 +f(q)g_{2,2}(q) + g_{2,2}(q)^2.\end{aligned}\] Since we may choose $g_{2,2}(q)$ freely in $\mathbb{C}$ and since the above can be viewed as a quadratic equation in $g_{2,2}(q)$, we may choose such a value that satisfies it. Letting $g_2 = g_{2,1}*_Ag_{2,2}$ gives us $g_2(q) = f(q)$ and $g_2(q^2) = f(q^2)$. 
    
    Proceeding in a similar fashion, we can construct $g_k$ using $k$ class-$A$ functions $g_{k,1}, \cdots, g_{k,k}$ such that $g_k(q^a) = f(q^a)$ for all $a \leq k$ and for all $q \in \Primes_A$. In the limit $g_k$ approaches $f$ and hence the closure of $\Rat(A)$ is $c(B)$.
\end{proof}

\begin{corollary}
    For all \Aname $A$-functions $A$ and for all $f \in c(\perf(A))$, $f$ may be approximated by a sequence $\{f_k\} \subseteq \Rat(A)$ such that for each $k$, $f_k$ is $A$-$k$-ic.
\end{corollary}

\begin{proof}
    The proof of Theorem \ref{close} was undertaken by constructing such a sequence, so we are done.
\end{proof}

\begin{corollary}
    Every multiplicative function $f$ may be approximated by a sequence of functions $f_n$, where each $f_n$ is $n$-ic. In particular, even though $\varphi = \mu*I$ is rational, $\varphi$ may be approximated by such a sequence of $n$-ic functions.
\end{corollary}

\begin{corollary}
    For all complete $A$-functions $A$, $(c(\perf(A)), *_{F(A)})$ and $(\Rat(F(A)),*_{F(A)})$ form groups.
\end{corollary}

The latter result is in fact true for any reflexive $A$-function which is associative -- hence, for example, $(\Rat(G),*_G)$ forms a (non-Abelian) group.

\begin{remark}
    It is unknown to the authors if results like this are known in the literature. Technically, one need not invoke the analytical language used here in order to realize Theorem \ref{close}, but the authors found no references to any such results in the literature.
    
    Indeed, naming functions which satisfy \eqref{eq:raft} ``rational'' seems to be even more fitting if one realizes that, under a natural discrete metric, these functions ``approximate'' the class of multiplicative arithmetical functions in a manner directly analogous to the rational approximation of a real number.
\end{remark}

\subsection{Examples}

We first extend the concept of the characteristic function $\chi_A$ to complete $A$-functions:

\begin{definition}
    Suppose $A$ is complete. Then for each $q \in \Primes_A$, let $\chi_A(q)$ denote the least $a$ such that the $A$-decomposition of $q^a$ is not $(q)^a$, or $\infty$ if no such $a$ exists.
\end{definition}

Note that this definition agrees with the original definition of $\chi_A$ from \cite{us2}.

\begin{definition}
    Let $A$ be complete. We say $A$ is \textbf{$N$-finitely primitive} if $\chi_A(q) \leq N$ for all $q \in \Primes_A$. Otherwise, if $\chi_A(q) < \infty$ for all $q \in \Primes_A$ but $A$ is not $N$-finitely primitive for any $N$, we call $A$ \textbf{quasi-infinitely primitive}. If $\chi_A(q) = \infty$ for some $q \in \Primes$, we call $A$ \textbf{infinitely primitive}.
    
    If $\chi_A(q) = N$ for all $q \in \Primes_A$, we call $A$ \textbf{strictly $N$-finitely primitive}, and if $\chi_A(q) = \infty$ for all $q \in \Primes_A$, we call $A$ \textbf{strictly infinitely primitive}.
\end{definition}

\begin{proposition}\label{eqorno}
    If $A$ is $N$-finitely primitive, then $\Rat(F(A)) = c(\perf(A))$, and furthermore for all $f \in c(\perf(A))$, $f$ may be represented as the $F(A)$-convolution of $N$ or fewer functons $f_k \in c(F(A))$. If $A$ is quasi-infinitely primitive or infinitely primitive, then $\Rat(F(A))$ is not equal to $c(\perf(A))$.
\end{proposition}

\begin{proof}
    In the proof for Theorem \ref{close}, if $A$ is $N$-finitely primitive then our approximations $f_k$ may stop at $k = N$ due to the fact that at each $q \in \Primes_A$, a system of equations with at most $N$ unknowns is to be solved. Hence, $Rat(F(A)) = c(P(A))$.
    
    By the same token if $A$ is quasi-infinitely primitive or infinitely primitive, then we may apply diagonal-type arguments to choose $f \in c(P(A))$ so as to avoid any finite sequence of $f_k$ that end in $f_N = f$, giving us our result. 
\end{proof}

The remainder of this work will be dedicated to a few examples of applying Theorem \ref{close} and Proposition \ref{eqorno} to various $A$-functions.

\begin{example}
    In the most basic application of Theorem \ref{close}, we have that the closure of $\Rat(D)$ is $c(D^1)$, but since $D$ is strictly infinitely primitive, $\Rat(D) \neq c(D^1)$, which is to say that not every multiplicative arithmetical function is rational.
\end{example}

\begin{example}
    Consider $\T_k$, the $k$-factorization $A$-functions of Litsyn and Shevelev. For $k = 2$, $\T_k = \T$ is perfect, but this is not the case for $k > 2$. Nevertheless, for each $k \in \N$, $\T_k$ is $k$-finitely primitive, so $\Rat(\T_k) = \perf(\T_k) = (\T_k)^1$.
\end{example}

\begin{example}
    In general, if $A$ is finitely structured, $A$ is \textit{infinitely} primitive, whereas if $A$ is infinitely structured, then $A$ may be quasi-infinitely primitive or $N$-finitely primitive. According to the class into which the structured $A$-function $A$ falls, we may have $\Rat(A) = c(A^1)$ or $\Rat(A) \neq c(A^1)$.
\end{example}

\begin{example}
    We shall construct a \Aname $A$-function which is strictly infinitely primitive but not multiplicative. We observe that, prior to this introduction, no such $A$-function with these properties exists -- $D$ is strictly infinitely primitive, but is multiplicative.
    
    Let $\S$ be the \Aname $A$-function such that $\Primes_{\S}$ is the set of all square-free integers, with the $\S$-decomposition of a non-square-free integer $n$ proceeding according to the following algorithm:
    \begin{enumerate}
        \item Let $a_1$ denote the largest power of $q_1 := \gamma(n)$ which divides $n$.
        \item Let $n_1 = \frac{n}{(\gamma(n))^{a_1}}$.
        \item Repeat steps 1 and 2 for $n := n_1$ until $n = 1$, incrementing the index each iteration.
        \item Write $n = q_1^{a_1} \cdots q_l^{a_l}$.
    \end{enumerate}
    Then $\S(n) := \S(q_1)^{a_1} \cdots \S(q_l)^{a_l}$. That this $A$-function is well-defined is clear, and by construction it is \Aname. 
    
    Furthermore, for all $q \in \Primes_{\S}$ and any $a \in \N$, $(q)^a$ is the $A$-decomposition of $q^a$, which implies that $\S$ is strictly infinitely primitive, as desired. Since $\S(12) = \{1,2,6,12\} \neq \{1,2,4\} \cdot \{1,3\}$, it is clear that $\S$ is not multiplicative.
\end{example}

\section{Acknowledgements}
    We would like to thank the Directed Research Program at The University of Texas at Dallas for facilitating the collaboration that led to this manuscript.

\end{document}